\documentclass[12pt]{article}
\usepackage{amsmath,amsfonts,amssymb,amsthm,graphicx}
\usepackage[arrow, matrix, tips]{xy}

\newtheorem*{theorem*}{Theorem}
\newtheorem{theorem}[subsubsection]{Theorem}
\newtheorem{cor}[subsubsection]{Corollary}
\newtheorem{lem}[subsubsection]{Lemma}
\newtheorem{prop}[subsubsection]{Proposition}

\newtheorem*{statement*}{Statement}

\theoremstyle{definition}
\newtheorem{defi}[subsubsection]{Definition}

\newtheorem{remark}[subsubsection]{Remark}

\numberwithin{equation}{section}

\newcommand{\cp}{{\mathbb{C}P}^n}
\newcommand{\Mod}{\mathcal{M}^d_g(\cp)}
\newcommand{\Mot}{\mathcal{MT}^d_g(\cp)}

\title{Stable topology of moduli spaces of holomorphic curves in $\cp$}

\author{David Ayala\footnote{This work was partially completed under the support of the ERC grant: ERC-AdG TMSS, no:228082}}

\begin{document}

\maketitle

\begin{abstract}
This paper is centered around a homotopy theoretic approximation to $\Mod$, the moduli space of degree $d$ holomorphic maps from genus $g$ Riemann surfaces into $\cp$.   There results a calculation of the integral cohomology ring $H^*(\Mod)$ for $*<<g<<d$.  The arguments follow those from a paper of G. Segal (\cite{segal:scanning}) on the topology of the space of rational functions.  

{\bf Note:
The present update points out an essential error in this work, which is detailed as Remark~\ref{error}.  
Any reader is invited to resolve, or circumvent, this issue.
}

\end{abstract}

\section*{Conventions}

Throughout this paper, all maps between topological spaces will be taken to be continuous and all diagrams commutative unless otherwise stated.  Sets of continuous maps are topologized with the compact open topology.  The symbol $H_q(-)$ will be used to denote the $q^{th}$ integral homology group of $-$.  

For $G$ a topological group acting continuously on a space $X$, the \textit{homotopy orbit space}, denoted $X//G$, is the standard orbit space $(EG\times X)/G$ where $EG$ is a chosen contractible space with a free action of $G$ and $G$ acts on the product by the diagonal action.  Observe that if the action of $G$ on $X$ is free then the projection $X//G \rightarrow X/G$ is a homotopy equivalence.

\section*{Introduction}

\subsection{Background and statement of main result}

A Riemann surface is a pair $(F_g,J)$ where $F_g\subset \mathbb{R}^\infty$ is a smooth oriented surface of genus $g$ embedded in Euclidean space, and $J$ is a complex structure on $F_g$ which agrees with the given orientation on $F_g$.  Write $\mathcal{J}_g$ for the space of such complex structures on $F_g$; a description of the topology of $\mathcal{J}_g$ is postponed to~\textsection\ref{sec:topology-of-J} where it is also shown to be contractible.

For $(F_g,J)$ a Riemann surface, write $Map^d( F_g , \cp )$ for the space of continuous maps from $F_g$ to $\cp$ having homological degree $d\in H_2({\mathbb{C}P}^n) \cong \mathbb{Z}$.  Write $Hol^d( (F_g,J) , \cp)$ for the subspace of $Map^d (F_g , \cp)$ consisting of holomorphic maps from $(F_g,J)$ to $\cp$.

\begin{theorem}[\cite{segal:scanning}]\label{thm:segal}
For $g>0$, the inclusion $Hol^d((F_g,J) , \cp) \hookrightarrow Map^d(F_g , \cp)$ induces an isomorphism in $H_q(-;\mathbb{Z})$ for $q < (d-2g)(2n-1)$.
\end{theorem}

This paper is devoted to proving such a theorem as the complex structure $J$ is allowed to vary.  This will amount to a statement about the moduli space of Riemann surfaces which will be defined presently.

For a fixed surface $F_g$, consider the set of pairs 
\begin{equation}\label{1}
\mathcal{J}^d_g({\mathbb{C}P}^n):=\{ ( J , h) \mid J \in \mathcal{J}_g \text{ and } h\in Hol^d( (F_g,J) , \cp ) \}.
\end{equation}
Topologize $\mathcal{J}^d_g({\mathbb{C}P}^n)\subset \mathcal{J}_g \times Map^d( F_g , \cp )$ as a subspace.
Let $Diff^+_g$ be the topological group of orientation preserving diffeomorphisms of $F_g$ endowed with the Whitney $C^\infty$ topology.  There is a continuous action of $Diff^+_g$ on $\mathcal{J}^d_g({\mathbb{C}P}^n)$ given by pulling back the complex structure and precomposing maps.  Define the (coarse) \textit{moduli space of degree $d$ genus $g$ holomorphic curves in $\cp$} as the resulting quotient 
\[
\Mod^c := \mathcal{J}^d_g({\mathbb{C}P}^n)/{Diff^+_g}.
\]
Similarly, consider the orbit space
\[
\Mot^c := (\mathcal{J}_g \times Map^d(F_g,\cp))/{Diff^+_g};
\]
referred to as the \textit{(coarse) topological moduli space of degree $d$ genus $g$ curves in $\cp$}.

More well-behaved are the homotopy orbit spaces
\[
\mathcal{M}^d_g(\cp):=\mathcal{J}^d_g(\cp)//Diff^+_g
\]
and
\[
\mathcal{MT}^d_g(\cp):=(\mathcal{J}_g\times Map^d(F_g,\cp))//Diff^+_g
\]
which will be referred to as the \textit{moduli space of degree $d$ genus $g$ curves in $\cp$}, and its topological counterpart, respectively.  
Note that the contracitbility of $\mathcal{J}_g$ implies the natural projection $\mathcal{MT}^d_g(\cp)\rightarrow Map^d(F_g,\cp)//Diff^+_g$ is a homotopy equivalence.

\begin{theorem}[Main Theorem]\label{thm:main-theorem}
The map 
\[
\mathcal{M}^d_g(\cp) \rightarrow \mathcal{MT}^d_g(\cp)
\]
induced by the natural inclusion induces an isomorphism in $H_q(-)$ for $q< (d-2g)(2n-1)$.
\end{theorem}

\begin{remark}
Theorem~\ref{thm:main-theorem} can be viewed as a statement about families.  Indeed, an isomorphism in $H_*$ implies an isomorphism in oriented bordism $MSO_*$.  As so Theorem~\ref{thm:main-theorem} says that any  $q$-dimensional family of continuous curves in $\cp$ is cobordant to a family of holomorphic curves provided $q<(d-2g)(2n-1)$.
\end{remark}

Recall the action of $Diff^+_g$ on $\mathcal{J}_g$.  Via Teichmuller theory, the isotropy subgroups $(Diff^+_g)_J$  are finite for each $J \in \mathcal{J}_g$. %??reference.  
Through a standard spectral sequence argument and the contractibility of $\mathcal{J}_g$, the product $EDiff^+_g \times \mathcal{J}_g$, with the diagonal action of $Diff^+_g$, becomes a \textit{rational} model for $EDiff^+_g$.  It follows that the projection maps 
\[
Map^d(F_g,\cp)//Diff^+_g\simeq  \mathcal{MT}^d_g(\cp) \rightarrow \Mot^c
\]
and
\[
\mathcal{M}^d_g(\cp)\rightarrow \mathcal{M}^d_g(\cp)^c
\]
induce isomorphisms on rational singular homology $H_*(-;\mathbb{Q})$.  There is an immediate corollary.

\begin{cor}
Induced by the obvious inclusion, the map 
\[
\mathcal{M}^d_g({\mathbb{C}P}^n)^c \rightarrow \mathcal{MT}^d_g({\mathbb{C}P}^n)^c
\]
induces an isomorphism in $H_q(-;\mathbb{Q})$ for $q<(d-2g)(2n-1)$.
\end{cor}

\begin{remark}
As described in~\textsection\ref{sec:topology-of-J}, the space $\mathcal{J}_g$ of complex structures on $F_g$ is canonically homeomorphic to the space of almost-complex structures on 
$F_g$.  In this way, the construction of  $\mathcal{MT}^d_g({\mathbb{C}P}^n)$ can be regarded as homotopy theoretic.  
\end{remark}

\begin{cor}
Induced by a zig-zag of maps of spaces is an isomorphism  
\[
H_q(\Mod) \cong
H_q(Map^d(F_g,\cp)//Diff^+_g)
\]
for $q<(d-2g)(2n-1).$
\end{cor}

There is the following useful group completion result due to Cohen and Madsen.

\begin{theorem}[\cite{cohen-madsen:background}]
For $X$ simply connected and for $h_*$ any connective homology theory, there is a map 
\[
Map(F_g,X)//Diff^+(F_g)\rightarrow \Omega^\infty({\mathbb{C}P}^\infty_{-1} \wedge X_+)
\]
which induces an isomorphism in $h_q$ for $q>(g-5)/2$.
\end{theorem}

\begin{cor}
Induced by a zig-zag of maps of spaces is an isomorphism 
\[
H_q(\Mod) \cong H_q(\Omega^\infty ({\mathbb{C}P}^\infty_{-1} \wedge {\mathbb{C}P}^n) )
\]
for $q<(d-2g)(2n-1) $ and $q<(g-5)/2$.
\end{cor}

The rational (co)homology of $\Omega^\infty$-spaces being reasonably well-understood, there is the following corollary which requires some notation to state.  For $V$ a graded vector space over $\mathbb{Q}$, denote by $A(V)$ the free graded-commutative $\mathbb{Q}$-algebra generated by $V$.  Let $\mathcal{K}$ be the graded vector space over $\mathbb{Q}$ generated by the set $\{k_i\}_{i\geq -1}$ where $\lvert k_i \rvert = 2i$.  Let $W$ another graded vector space over $\mathbb{Q}$.  Denote by $(\mathcal{K}\otimes W)_+$ the positively graded summand of the vector space $\mathcal{K}\otimes W$.  Recall that $H^*({\mathbb{C}P}^n) \cong \mathbb{Q}[c]/{c^{n+1}}$ where $\lvert c \rvert=2$.

\begin{cor}
There is an isomorphism of graded rings
\[
H^*(\Mod;\mathbb{Q})\cong A((\mathcal{K}\otimes \mathbb{Q}[c])_+).
\]
through the range $*<(d-2g)(2n-1)$ and $q<(g-5)/2$.
\end{cor}

 \subsection{Variants of the main theorem}
 
There are two easy variations of the main theorem.  See~\textsection\ref{variants} for precise definitions.  
 
 Write $\mathcal{M}^d_{g,k}({\mathbb{C}P}^n)$ for the moduli space of degree $d$ curves in ${\mathbb{C}P}^n$ with $k$ marked points.  Write $\mathcal{MT}^d_{g,k}({\mathbb{C}P}^n)$ for the topological counterpart.

 \begin{theorem}\label{thm:marks}
The standard map
\[
\mathcal{M}^d_{g,k}({\mathbb{C}P}^n) \rightarrow \mathcal{MT}^d_{g,k}({\mathbb{C}P}^n)
\]
induces an isomorphism in $H_q(-)$ for $q<(d-2g)(2n-1)$.
\end{theorem}

Fix a finite collection of marked points $p_1,\dots,p_k\in F_g$ and an equivalence relation $\sim$ on $\{p_i\}$.  Denote this data by $F:=(F_g,\{p_i\},\sim)$.  Denote the cardinality $n(F):=\lvert \{p_i\}_{/\sim}\rvert$ and the number $g(F)$ such that 
\[
\chi((F_g)_{/\sim}) = 2-2g(F) + n(F).
\]
Write $\mathcal{M}^d_{[F]}({\mathbb{C}P}^n)$ for the moduli space of degree $d$ genus $g$ holomorphic marked curves $[(J,h)]$ in ${\mathbb{C}P}^n$ satisfying $h(p_i) = h(p_j)$ when $p_i\sim p_j$.  Write $\mathcal{MT}^d_{[F]}({\mathbb{C}P}^n)$ for the topological counterpart.

\begin{theorem}\label{thm:singular-main-theorem}
The standard map
\[
\mathcal{M}^d_{[F]}(\cp)\rightarrow \mathcal{MT}^d_{[F]}(\cp)\]
induces an isomorphism in $H_q(-)$ for $q<(d-2(g(F)-n(F)+1))(2n-1)$.
\end{theorem}

\section{Preliminaries}

\subsection{The topology of $\mathcal{J}_g$}\label{sec:topology-of-J}

Write $Fr_2(\mathbb{R}^N):= Emb_{lin}(\mathbb{R}^2,\mathbb{R}^N)$ for the space of linear embeddings, topologized with the compact-open topology.  There is a continuous free action of the orientation preserving general linear group $GL^+_2(\mathbb{R})$ on $Fr_2(\mathbb{R}^N)$ given by precomposition.  The resulting quotient space is the oriented Grassmann $Gr^+_2(\mathbb{R}^N)$ of oriented $2$-planes in $\mathbb{R}^N$.  

Consider the set (space) $\mathcal{J}_{lin}(\mathbb{R}^2)$ of linear endomorphisms $J:\mathbb{R}^2 \to \mathbb{R}^2$ such that $J^2 = -id$ and such that, for $v\neq 0$, the determinant $det(v\mid Jv) >0$.  Topologize $\mathcal{J}_{lin}(\mathbb{R}^2)$ with the compact-open topology.  There is a continuous action of $GL^+_2(\mathbb{R})$ on $\mathcal{J}_{lin}(\mathbb{R}^2)$ given by conjugation $(g,J)\mapsto g\circ J\circ g^{-1}$.  This action is transitive and has isotropy subgroup $GL_1(\mathbb{C})$.  Define the Grassmann $Gr^\mathcal{J}_2(\mathbb{R}^N)$ of  \textit{complex} $2$-planes in $\mathbb{R}^N$ as the orbit space $Fr_2(\mathbb{R}^N) \times_{GL_2(\mathbb{R})} \mathcal{J}_{lin}(\mathbb{R}^2)$.  The map $\mathcal{J}_{lin}(\mathbb{R}^2) \to *$ induces a fibration $Gr_2^\mathcal{J}(\mathbb{R}^N) \to Gr_2^+(\mathbb{R}^N)$.  The fibers of this fibration are homeomorphic to $GL^+_2(\mathbb{R})/GL_1(\mathbb{C})$.  The inclusion $GL_1(\mathbb{C}) \hookrightarrow GL_2^+(\mathbb{R})$ is a deformation retract, and it follows that the fibers of this fibration are contractible.

There are obvious directed systems and a morphism between them
\[
\xymatrix{
\dotsi  \ar[r] 
&
Gr^\mathcal{J}_2(\mathbb{R}^N)  \ar[d]  \ar[r]  
&
Gr^\mathcal{J}_2(\mathbb{R}^{N+2})  \ar[d]  \ar[r]
&
\dots
\\
\dotsi   \ar[r]
&
Gr^+_2(\mathbb{R}^N)  \ar[r]
&
Gr^+_2(\mathbb{R}^{N+2})  \ar[r]
&
\dots ,
}
\]
where the horizontal arrows induced by the inclusions $\mathbb{R}^N \cong \mathbb{R}^N\times\{0\} \hookrightarrow \mathbb{R}^N\times\mathbb{R}^2 =\mathbb{R}^{N+2}$.  Denote the morphism of colimits of these directed systems as $\theta:BU(1) \to BSO(2)$.  This map $\theta$ is a fibration with contractible fibers.

The embedding $F_g\hookrightarrow\mathbb{R}^\infty$ together with the orientation $\sigma$ of $F_g$ determines a map $F_g\xrightarrow{\tau_F}BSO(2)$, given by $p\mapsto (T_pF_g\subset \mathbb{R}^\infty , \sigma_p)$, which classifies the tangent bundle $\tau_{F_g}$ of $F_g$.  An \textit{almost-complex structure} on $F_g$ is a map $J$ in the commutative diagram
\[
\xymatrix{
&
BU(1) \ar[d]^\theta
\\
F_g  \ar[ur]^J  \ar[r]^{\tau_F}
&
BSO(2).
}
\]
Denote by $\mathcal{J}_g\subset Map(F_g, BU(1))$ the subspace of almost-complex structures on $F_g$.  Note that $\mathcal{J}_g$ is in bijection with the set of bundle automorphisms $J:\tau_{F_g} \xrightarrow{\cong} \tau_{F_g}$ such that $J^2 = -id$.  For dimension reasons, the Nijenhuis tensor on $F_g$ will always vanish (see \cite{newlander-nirenberg:complex-almost-complex}).  Therefore $\mathcal{J}_g$ is \textit{equal} to the set (space) of complex structures.  
Finally, because $\theta$ is a fibration with contractible fibers, the space $\mathcal{J}_g$ of complex structures on $F_g = (F_g,\sigma)$ is contractible.

For $d>0$ a positive integer, the topology just defined on $\mathcal{J}_g$ determines the subspace topology on the set $\mathcal{J}^d_g({\mathbb{C}P}^n)$ defined in~(\ref{1}).
Likewise, for $x_0\in F_g$ a base point and $n=1$, we obtain a topology on the subspace 
\[
\mathcal{J}^d_g({\mathbb{C}P}^1)_\ast
:=
\{ (J,h) \mid h(x_0) = \infty \in \mathbb{C}P^1\}
\subset
\mathcal{J}^d_g({\mathbb{C}P}^1)   .
\]
With respect to this topology, the projection 
\begin{equation}\label{2}
\mathcal{J}^d_g({\mathbb{C}P}^1)_\ast
\to 
\mathcal{J}_g
\end{equation}
is continuous.

\subsection{Symmetric products}

Let $X$ be a topological space.  
For each $d\geq 0$ there is a continuous action of the permutation group $\Sigma_d$ on the product $X^{\times d}$ given by permuting the factors.  Define the \textit{$d$-fold symmetric product of $X$} to be the quotient topological space
\[
Sp_d(X):= X^{\times d}/\Sigma_d.
\]
Define the \textit{symmetric product of $X$} as the disjoint union $Sp(X):= \amalg_{d\in\mathbb{N}} Sp_d(X)$.  Given a base point $*\in X$, define $Sp(X,*):= Sp(X)/\sim$ where $\sim$ is the equivalence relation generated by $[(*,x_2,\dots,x_d)]\sim[(x_2,\dots,x_d)]$.

\subsection{Strategy for proving the main theorem}

The argument for the proof of Theorem~\ref{thm:main-theorem} will follow that of Segal's (\cite{segal:scanning}) for when the complex structure is fixed.  Details will be supplied for $n=1$ in which case the target complex manifold is ${\mathbb{C}P}^1 \approx S^2$.  The general situation is not much more difficult as will be outlined later.

The idea is to regard a holomorphic map $(F_g,J)\rightarrow {\mathbb{C}P}^1$ as a rational function on $(F_g,J)$, then to regard a rational function as a pair of divisors $(\eta,\xi)$ given by its zeros and poles.  The degree to which such a pair of divisors is realized in this way from a rational function is described by a theorem of Abel's.   Abel's theorem results in a map from the space of divisors on $(F_g,J)$ to the Jacobian of $(F_g,J)$ whose fiber is the space of rational functions on $(F_g,J)$.  This Jacobian is then identified with a standard torus which is independent of $J$.  The resulting sequence is a \textit{homology fibration through a range}.

Using `scanning maps', there is a comparison of this homology fibration to a homotopy theoretic fibration with fiber $Map(F_g,S^2)$.  These scanning maps are shown to be equivalences from which it follows that the space of pairs $(h,J)$, where $h$ is a rational function on $(F_g,J)$, is homology equivalent to $Map(F_g,S^2)$ through a range.  With sufficient care, one has a family of such constructions parametrized by the space of complex structures.  A simple spectral sequence argument is then in place to have a similar comparison on homotopy quotients by $Diff^+_g$ and the result follows.

\section{Spaces of divisors}

This section is analogous to the discussions in~\textsection3 preceding Proposition~3.1 and that following Proposition~4.5 of~\cite{segal:scanning}; here, the complex structure of $F_g$ is allowed to vary.

Once and for all, choose a base point $x_0\in F_g$.  Write $Sp_d:= Sp_d(F_g\setminus x_0)$ and $Sp:= \coprod_d Sp_d$.  
There is a bijection with the underlying set of $Sp_d$ and the set of positive degree $d$ divisors on $F_g\setminus x_0$.  Define 
\[
Div_d:=\{(\eta,\zeta)\mid \eta\cap\zeta = \emptyset\}\subset Sp_d\times Sp_d
\]
to be the space of pairs $(\eta,\zeta)$ of \textit{disjoint} positive divisors on $F_g \setminus x_0$ of bi-degree $(d,d)$.  
Define 
\[
Div:=\coprod_d Div_d .
\]
The remainder of this section will be devoted to defining a larger space $Div\hookrightarrow \widehat{Div}$ which is better homotopically behaved.

Fix some positive integer $d_0$ then choose a continuous section
\begin{equation}\label{e1}
f\colon 
\mathcal{J}_g 
\longrightarrow
\mathcal{J}^{d_0}_g({\mathbb{C}P}^1)_\ast   
\end{equation}
of the projection~(\ref{2}).
%such that for each $J\in\mathcal{J}_g$ , the holomorphic map $f_J:(F_g,J)\to {\mathbb{C}P}^1$ has a single pole at $x_0\in F_g$ of multiplicity $d_0$.   
For each $J\in \mathcal{J}_g$, and for each $z\in\mathbb{C}\subset{\mathbb{C}P}^1$, the preimage $f_J^{-1}(z)\subset F_g\setminus \{x_0\}$ is a positive divisor of degree $d_0$.  For each $J\in\mathcal{J}_g$, there is a large number $M_J$ such that for $\lvert z\rvert \geq M_J$, the positive divisor $f_J^{-1}(z)$ consists of $d_0$ \textit{distinct} degree one divisors.  Choose a continuous such assignment $J\mapsto M_J\in [2,\infty)\subset {\mathbb{C}P}^1$.

\begin{remark}\label{error}
The previous paragraph postulates the existence of a certain section $f$ of the projection
$
\mathcal{J}^{d_0}_g({\mathbb{C}P}^1) 
\to
\mathcal{J}_g   ,  
$
even just for large enough $d_0$.
The existence of such a section is not obvious: I have not been able to justify this assertion, and believe it to be wrong.  
This section $f$ is essential for the construction of the topological space $\widehat{Div}$, as it is defined below.  
As this topological space $\widehat{Div}$ is critical for the logic of the main results of this paper, this issue invalidates the main results of this paper.  

\end{remark}

\begin{remark}\label{independent}
This remark records an observation exploited in~\textsection\ref{abel} where the notation is explained.  Because $f_J$ is holomorphic, the integration $I(f_J^{-1}(z)) = I(f_J^{-1}(z')) \in T_J$ are identical for any $z,z'\in \mathbb{C}$.  
\end{remark}

Regard $\mathbb{N} = \{0<1<\dots\}$ as a category in the standard way.  
For $d$ fixed, and for $d_0$ as in~(\ref{e1}), consider the functor $\mathcal{D}:\mathbb{N}\rightarrow \mathsf{Top}_{/\mathcal{J}_g}$ into topological spaces over $\mathcal{J}_g$ given on objects constantly as
\[
\mathcal{D}(n) = \mathcal{J}_g\times Div;
\]
and given on generating morphisms $n< n+1$ as $\iota= \iota_n:\mathcal{D}(n) \to \mathcal{D}(n+1)$ given by
\[
(J,\eta,\zeta)\mapsto (J,\eta + y_J(\eta,\zeta)~,~ \zeta + z_J(\eta,\zeta))
\]
where 
\[
y_J(\eta,\zeta) = f_J^{-1}(1+ max\{M_J , \lvert f_J(x)\rvert \mid (x\in \eta\cup \zeta)\} )
\]
and
\[
z_J(\eta,\zeta) = f_J^{-1}(-1 - max\{M_J , \lvert f_J(x)\rvert \mid (x\in \eta\cup \zeta)\}).
\]
Denote the colimit $\widehat{Div} = colim~ \mathcal{D}$.  
The projection $\widehat{Div}\rightarrow \mathcal{J}_g$ is a fibration.  One can think of a point in the fiber $\widehat{Div}(J)$ over $J\in\mathcal{J}_g$ as a pair of disjoint ``infinite'' positive divisors on the Riemann surface $(F_g\setminus x_0,J)$ whose difference with some pair $(\sum y_n, \sum z_n)$ is a pair of finite divisors.

Denote the subspace
\[
\mathcal{D}_d(n):= (\mathcal{J}_g\times Div_{d+ nd_0})\subset \mathcal{D}(n).
\]
For each $J\in\mathcal{J}_g$, the embedding
$\iota: \mathcal{D}_d(n)\hookrightarrow \mathcal{D}_d(n+1)$ given by $(J,\eta,\zeta)\mapsto (\eta+ y_J(\eta,\zeta) , \zeta + z_J(\eta,\zeta))$ 
has \textit{trivial normal bundle} meaning that it extends to an open embedding over $\mathcal{J}_g$ of $\mathcal{D}_d(n-1)\widetilde{\times} (V_{y_J(\eta,\zeta)}\times V_{z_J(\eta,\zeta)})$ where for $S\subset F_g$ a finite subset, $V_S$ is a sufficiently small tubular neighborhood of $S\subset F_g$.  Denote by $\mathcal{D}\supset \mathcal{D}_d:\mathbb{N}\rightarrow \mathsf{Top}$ the subfunctor given by $n\mapsto \mathcal{D}_d(n)$.  Write 
\[
\widehat{Div}_d \subset \widehat{Div} 
\]
for the component which is the colimit of $\mathcal{D}_d$.

\begin{remark}
The idea for introducing $\widehat{Div}_d =: \widehat{Div}_d(F_g)$ rather than being satisfied with $\mathcal{J}_g\times Div_d$ is that $\widehat{Div}(-)$ behaves better on non-compact arguments.  Namely, restriction maps among $\widehat{Div}(-)$ which one expects to be quasifibrations from Dold-Thom theory are indeed quasifibrations.
\end{remark}

\section{The homology fibration}

\subsection{Phrasing Abel's Theorem}\label{abel}

Let $V_J$ be the space of holomorphic 1-forms on the Riemann surface $(F_g,J)$.  There is the natural inclusion $H_1(F_g;\mathbb{Z})\hookrightarrow {V_J}^*$ as a non-degenerate lattice yielding the $g$-torus 
\[
 T_J := {V_J}^*/H_1(F_g;\mathbb{Z})
\]
known as the \textit{Jacobian variety} of the Riemann surface $(F_g,J)$.

Consider the space of pairs 
\[
\mathcal{J}_g\tilde{\times}T_J  :=  \{(J,v)\mid v\in T_J\}.
\]
Projection onto the first coordinate makes $\mathcal{J}_g\tilde{\times}T_J \rightarrow \mathcal{J}_g$ into a fiber bundle whose fiber over any $J\in\mathcal{J}_g$ is the Jacobian variety $T_J$.

For each $p\in F_g$, choose a piecewise smooth path $\gamma:[0,1]\rightarrow F_g$ from $p$ to $x_0$.  Integration along such $\gamma$ gives a map $I:Div_d\rightarrow T_J$.  Explicitly, given a positive divisor $\eta=\sum m_i p_i$ of $F_g$ and paths $\gamma_i$ from $p_i$ to $x_0$, write $\int_{\gamma_\eta}$ for $\sum m_i\int_{\gamma_i}$.  Define
\begin{equation}\label{formula}
I(\eta,\zeta)=(\alpha\mapsto (\int_{\gamma_\eta}-\int_{\gamma_\zeta}) \alpha).
\end{equation}
Note that the same formula extends $I$ to a map $I:Sp(F_g)\times Sp(F_g)\rightarrow T_J$.

Interpreting Abel's theorem (\cite{griffiths-harris:principles-of-algebraic-geometry}, p. 231, for a general reference), the fiber of $I$ over $0\in T_J$ is the space of degree $d$ meromorphic functions on $(F_g,J)$, that is, $Hol^d_*((F_g,J),{\mathbb{C}P}^1)$.  Here, the $*$ denotes based maps.

\begin{defi}
A map $p:E\rightarrow B$ is a \textit{homology fibration up to degree $k$} if for each $b\in B$ the inclusion $fiber_b(p)\rightarrow hofiber_b(P)$ induces an isomorphism in $H_q(-)$ for $q<k$.  
\end{defi}

\begin{remark}
This definition is similar to that in~\cite{mcduff-segal:group-completion} where they introduce the notion of a \textit{homology fibration}.
\end{remark}

Through out the paper, we will implicitly make use of the following 
\begin{prop}
Let $E\xrightarrow{p}B$ be a continuous map with $B$ paracompact and locally contractible.  Then $p$ is a \textit{homology fibration up to degree $k$} if and only if for each $b\in B$ there is a contractible neighborhood $b\in U\subset B$ such that $p^{-1}(b) \hookrightarrow p^{-1}(U)$ induces an isomorphism in $H_q(-)$ for $q < k$.  
\end{prop}

\begin{proof}
The proof is nearly identical to that of Proposition~5 in~\cite{mcduff-segal:group-completion}.
\end{proof}

Segal shows that  $I: Div_d\rightarrow T_J$ is a \textit{homology fibration up to degree $d-2g$}.  In this paper, we are interested in a similar statement as the complex structure $J$ is allowed to vary.  This is accomplished as follows.

There is a natural inclusion 
$V_J^*\hookrightarrow H_1(F_g;\mathbb{C})$.
This along with the projection $\mathbb{C}\rightarrow\mathbb{R}$ yields the canonical homeomorphism
\[
V_J^*\cong H_1(F_g;\mathbb{R})
\]
which is $H_1(F_g;\mathbb{Z})$-eqivariant.  This results in a canonical homeomorphism of $g$-tori
\[
T_J\cong H_1(F_g;\mathbb{R})/H_1(F_g;\mathbb{Z}).
\]
Denote this standard $g$-torus $H_1(F_g;\mathbb{R})/H_1(F_g;\mathbb{Z})$ by $T_0$.
There results an isomorphism of fiber bundles 
\[
\xymatrix{
\mathcal{J}_g\tilde{\times} T_J \ar[r]^\cong \ar[d]  
&  
\mathcal{J}_g\times T_0 \ar[d] 
\\
\mathcal{J}_g \ar[r]^{id}  
& 
\mathcal{J}_g.
}
\]

Define $\mathcal{J}^d_g({\mathbb{C}P}^1)\subset \mathcal{J}_g\times Map(F_g,\cp)$ to be the subspace consisting of those pairs $(J,h)$ for which $h$ is $J$-holomorphic.  Denote subspace $\mathcal{J}^d_g({\mathbb{C}P}^1)_*\subset \mathcal{J}^d_g({\mathbb{C}P}^1)$ consisting of those pairs $(J,h)$ for which $h$ is a based map.
Using Abel's theorem, the fiber of the composite 
\[
\mathcal{I}: \mathcal{J}_g\times Div_d \xrightarrow{id\times I} \mathcal{J}_g\tilde{\times} T_J \xrightarrow{\cong} \mathcal{J}_g\times T_0 \xrightarrow{pr} T_0 
\]
is the space $\mathcal{J}^d_g({\mathbb{C}P}^1)_*$.  Extend notation and write $\mathcal{I}:\mathcal{J}_g\times Sp\times Sp\rightarrow T_0$ for the map given by the same formula~(\ref{formula}).

\subsection{The Homology Fibration}

The following theorem is the crux of the paper and is the most technical argument presented.  The argument follows that in the proof of Proposition~4.5 of~\cite{segal:scanning}, here, of course, the comlpex structure on $F_g$ is regarded as an additional parameter.

\begin{theorem}\label{thm:homology-fibration}
The map
$\mathcal{I}:\mathcal{J}_g\times Div_d \rightarrow T_0$ is a homology fibration up to degree $d-2g$.  
\end{theorem}

\begin{proof}
Through out this proof, we will suppress the subscript $g$ wherever is should appear.

The product $P_d:=Sp_d\times Sp_d$ is stratified by the spaces $P_{d,k}:=\{(\eta,\zeta)\mid deg(\eta \cap \zeta)\geq k\}$:
\[
Sp_d\subset P_{d,d-1}\subset...\subset P_{d,1}\subset P_d = Sp_d\times Sp_d.
\]
Notice that the inclusion $Div_{d-k}\times Sp_k \hookrightarrow P_{d,k}$, given by $((\eta,\zeta),\xi) \mapsto (\eta +\xi , \zeta + \xi)$, is a homeomorphism onto the open stratum $P_{d,k}\setminus P_{d,k+1}$.

Extend the path-integral map $\mathcal{I}$ from above to each $\mathcal{J}\times P_{d,k}$ in the apparent way.  Let $S\subset T_0$  be a subset.  Denote $(\mathcal{J}\times P_{d,k})_S := \mathcal{I}^{-1}(S) \subset \mathcal{J}\times P_{d,k}$.   There is a sequence of inclusions
\[
(\mathcal{J}\times Sp_d)_S \subset (\mathcal{J}\times P_{d,d-1})_S  \subset\dots\subset (\mathcal{J}\times P_{d,1})_S \subset (\mathcal{J}\times P_d)_S.
\]
with $(\mathcal{J}\times P_{d,k})_S \setminus (\mathcal{J}\times P_{d,k+1})_S  \cong (\mathcal{J}\times (Div_{d-k} \times Sp_k))_S$.  For $v\in W \subset T_0$ a contractible neighborhood of $v$, let $j_{-}$ denote any such inclusion $(-)_v\hookrightarrow (-)_W$.  
Unwinding definitions, the proof of the lemma amounts to showing $H_r(j_{\mathcal{J}\times Div_d}) = 0$ when $r< d-2g$.

The inclusion $Div_{d-k}\times Sp_k\hookrightarrow P_{n,k}$ is an open embedding.  The inclusion $P_{d,k}\hookrightarrow P_{d,k-1}$ has normal bundle denoted $\nu_{P_{d,k}}$.  It is thus possible to form the well-behaved normal bundle $\nu_{\mathcal{J}\times P_{d,k}}$ of the embedding $\mathcal{J}\times P_{d,k}\hookrightarrow \mathcal{J}\times P_{d,k-1}$, namely, 
\[
\nu_{\mathcal{J}\times P_{d,k}}:= pr^* \nu_{P_{d,k}}
\]
where $pr:\mathcal{J}\times P_{d,k} \rightarrow P_{d,k}$ is projection onto the second factor.  There is an analogous normal bundle for the two embeddings $(\mathcal{J}\times P_{d,k})_{v,W}\hookrightarrow (\mathcal{J}\times P_{d,k-1})_{v,W}$.

Proceeding by (downward) induction on $d$, assume for $0< k\leq d$ that $H_r(j_{\mathcal{J}\times Div_{d-k}\times Sp_k})=0$ when $r+k\leq d-k-2g$.  That is to say $j_{\mathcal{J}\times Div_{d-k}\times Sp_k}$ induces an isomorphism in $H_r$ for $r<d- k-2g$ and a surjection in $H_{d-k-2g}$.  Refer to this inductive hypothesis as the \textit{primary} inductive hypothesis.  We wish to prove the case $k=0$.

Consider the diagram~(\ref{eq:sequence}) of exact sequences of homology groups associated to the pair 
\[
(\mathcal{J}\times P_{d,k}~,~\mathcal{J}\times (P_{d,k}\setminus P_{d,k+1})) \cong (\mathcal{J}\times P_{d,k}~,~\mathcal{J}\times Div_{d-k}\times Sp_k).  
\]
From the above discussion, the relative term is homotopy equivalent to the Thom space $Th(\nu_{\mathcal{J}\times P_{d,k+1}})$ of the normal bundle.

\begin{equation}\label{eq:sequence}
\xymatrix{
\dotsi \ar[r]
&
H_{r-1}Th(\nu_{(\mathcal{J}\times P_{d,k+1})_v})   \ar[r] \ar[d]^j 
&
H_r(\mathcal{J}\times Div_{d-k}\times Sp_k)_v  \ar[d]^j \ar[r] 
&
H_r(\mathcal{J}\times P_{d,k})_v   \ar[d]^j   \ar[r]  
&
\dotsi
%H_r Th(\nu_{(\mathcal{J}\times P_{d,k+1})_v})   \ar[d]  \ar[r]  
%&
%H_{r+1}(\mathcal{J}\times Div_{d-k}\times Sp_k)_v  \ar[d]   
\\
\dotsi \ar[r]
&
H_{r-1}Th(\nu_{(\mathcal{J}\times P_{d,k+1})_W})  \ar[r]  
&
H_r(\mathcal{J}\times Div_{d-k}\times Sp_k)_W  \ar[r] 
&
H_r(\mathcal{J}\times P_{d,k})_W    \ar[r]  
&
\dotsi
%H_r Th(\nu_{(\mathcal{J}\times P_{d,k+1})_W}  )   \ar[r]  
%&
%H_{r+1}(\mathcal{J}\times Div_{d-k}\times Sp_k)_W. 
}
\end{equation}

Invoke a nested \textit{secondary} (downward) induction argument on $0<k\leq d$ to assume $H_r(j_{\mathcal{J}\times P_{d,l}})=0$ for $r\leq d-l$ with $k<l\leq d$.  For the moment, assume the base case $k=d$ of this secondary induction hypothesis.  Using the primary inductive hypothesis, the Thom isomorphism, and the 5-lemma on~(\ref{eq:sequence}), it follows that $H_r(j_{\mathcal{J}\times P_{d,k}})=0$ for $r\leq d-k$ ($k>0$).  The case $k=d$ is easy enough as outlined by the following two facts.

Firstly, the Riemann-Roch formula and Abel's theorem tells us that, for $d> 2g$, $I:\{J\}\times Sp_d\rightarrow T_0$ is a fiber bundle having fiber the $d-g$ dimensional complex vector space of degree $d$ (based) meromorphic functions on $(F_g,J)$.  Secondly, as we allow variation in $J\in \mathcal{J}\simeq *$, the map $\mathcal{I}: \mathcal{J}\times Sp_d\rightarrow T_0$ is a fiber bundle with fiber identified with the product $\mathcal{J}\times \mathbb{C}^{d-g}$.  These same two facts, imply that $H_r(j_{\mathcal{J}\times P_d})=0$ since $P_d=Sp_d\times Sp_d$. 

Considering diagram~(\ref{eq:sequence}) for $k=0$, the 5-lemma tells us that $H_r(j_{\mathcal{J}\times Div_d})=0$ provided $d\geq 2g$ and $r\leq d-2g$.  This completes the inductive step.

\end{proof}

Let $(J,\eta,\zeta)\in \mathcal{D}_d(n)$.  As observed in Remark~\ref{independent}, the diagram
\begin{equation}\label{fibers}
\xymatrix{
\mathcal{D}_d(n)  \ar[r]^\iota  \ar[d]^{\mathcal{I}}
&
\mathcal{D}_d(n+1)  \ar[d]^{\mathcal{I}}
\\
T_0 \ar[r]^=
&
T_0
}
\end{equation}
commutes.
There results a universal map from the colimit 
\[
\mathcal{I}:\widehat{Div}_d\rightarrow T_0.  
\]

The next task is to prove the following Theorem.  The Theorem and its proof are analogous to Theorem~5.1 of~\cite{segal:scanning} and its proof.  Again, here the complex structure is allowed to vary.
\begin{theorem}\label{lem:divisors}
Let $d,n\in\mathbb{N}$.  The induced map on fibers in the above diagram~(\ref{fibers}) induces an isomorphism in $H_q(-)$ for $q<d-2g$.
\end{theorem}

\begin{proof}
To prove Theorem~\ref{lem:divisors} it will suffice to show $\iota: \mathcal{D}_d(n) \rightarrow \mathcal{D}_d(n+1)$ induces an isomorphism in $H_q(-)$ for $q<d-2g$.  This is indeed sufficient using the Zeeman comparison theorem (\cite{zeeman:comparison}) for the apparent Leray-Serre spectral sequences (see~\cite{mccleary:spectral-sequences},~\textsection5).  Recall that the canonical inclusion $\mathcal{D}_d(n) \hookrightarrow \mathcal{J}_g\times Div_{d+nd_0}$ is a homotopy equivalence.  Choose once and for all a homotopy inverse $\mathcal{J}_d\times Div_{d+nd_0}\rightarrow \mathcal{D}_d(n)$ for each $n\in\mathbb{N}$ and denote the resulting map again as
\begin{equation}\label{iota}
\iota: \mathcal{J}_g\times Div_d\rightarrow\mathcal{J}_g\times Div_{d+nd_0}
\end{equation}

It turns out that most of the work has already been done.  As in the proof of Theorem~\ref{thm:homology-fibration}, consider the sequence of closed inclusions
\begin{equation}\label{stratification}
(\mathcal{J}_g\times Sp_d)\subset (\mathcal{J}_g\times P_{d,d-1})\subset...\subset (\mathcal{J}_g\times P_{d,1})\subset (\mathcal{J}_g\times P_d) = (\mathcal{J}_g\times Sp_d\times Sp_d).
\end{equation}
In the natural way, extend the map $\iota$ in~(\ref{iota}) to a map to a map of sequences~(\ref{stratification}).
In the proof of Theorem~\ref{thm:homology-fibration}, we compared the $H_*$-long exact sequences for the pairs $((P_{d,k})_v,(Div_{d-k}\times Sp_k)_v)$ and $((P_{d,k})_W,(Div_{d-k}\times Sp_k)_W)$.  Here we compare the $H_*$-long exact sequences of the pairs $((P_{d,k})_v,(Div_{d-k}\times Sp_k)_v)$ and $((P_{d+1,k})_v,(Div_{d+1-k}\times Sp_k)_v)$.  The appropriate nested induction argument here is nearly identical to that of Theorem~\ref{thm:homology-fibration}.  Again, the essential fact comes from Abel's theorem which is that both $\mathcal{J}_g\times Sp_d \rightarrow T_0$ and $\mathcal{J}_g\times Sp_d\times Sp_d\rightarrow T_0$ are fiber bundles for $d\geq 2g$.  Details are left to the interested reader.
\end{proof}

\begin{cor}
The integration map
\[
\mathcal{I}: \widehat{Div}_d  \rightarrow T_0
\]
is a homology fibration up to degree $d-2g$.
\end{cor}

\begin{proof}

It follows from Theorem~\ref{thm:homology-fibration}  and Proposition 3 of~\cite{mcduff-segal:group-completion} that the resulting map of telescopes 
\[
\mathcal{I}: (hocolim ~\mathcal{D}_d) \rightarrow hocolim(T_0\xrightarrow{=} T_0\xrightarrow{=}\dots)
\]
is a homology fibration up to degree $d-2g$.  The claim follows from the observation that $\mathcal{D}(n)\hookrightarrow \mathcal{D}(n+1)$ extends to an open inclusion of a trivial bundle over $\mathcal{D}(n)$, and is thus a cofibration.  
\end{proof}

\section{Comparing Fibration Sequences}

\subsection{The comparison fibration}

This subsection is an exposition of \textsection2 of~\cite{segal:scanning}.

Theorem~\ref{thm:homology-fibration} thus gives the sequence
\begin{equation}\label{eq:homology-fibration} 
\mathcal{J}^d_g({\mathbb{C}P}^n)_* \rightarrow \mathcal{J}_g\times Div_d \xrightarrow{\mathcal{I}} T_0
\end{equation}
as a homology fibration up to degree $d-2g$.  The idea now is to compare this sequence with a homotopy theoretic sequence whose fiber is $\mathcal{J}_g\times Map^d_*(F_g,S^2)$.  For this, consider the (homotopy) fibration sequence 
\begin{equation}\label{eq:fibration}
S^2\rightarrow {\mathbb{C}P}^\infty\vee{\mathbb{C}P}^\infty\rightarrow{\mathbb{C}P}^\infty
\end{equation}
which is defined as follows.

The topological group $S^1$ acts on $S^2$ by rotating $S^2$ fixing the north and south poles.  Associated to this $S^1$-action there is a fibration sequence 
\begin{equation}\label{S-fibration}
S^2\rightarrow ES^1\times_{S^1} S^2 \rightarrow BS^1.
\end{equation}
Regard the total space of this fiber bundle as a union of two disk bundles, one corresponding to the northern hemisphere of $S^2$, the other to the southern; the two disk bundles are glued together along their fiber-wise boundary $S^1$-bundle. The total space of each disk-bundle is homotopy equivalent to $BS^1$.  The equatorial $S^1$-bundle is a model for $ES^1$ and is thus contractible.  The fibration seuqnce~(\ref{S-fibration}) then becomes the (homotopy) fibration sequence
\begin{equation}\label{topological-fibration}
S^2\rightarrow BS^1\vee BS^1 \rightarrow BS^1.
\end{equation}
Lastly, recall that ${\mathbb{C}P}^\infty$ is a model for $BS^1$.

\subsection{The Scanning Map}

Our goal now is to define a `scanning map' 
\[
S:\widehat{Div} \rightarrow \mathcal{J}\times Map_*(F_g,{\mathbb{C}P}^\infty\vee {\mathbb{C}P}^\infty).
\]

Regard ${\mathbb{C}P}^\infty$ as the projectivization of the vector space $\mathbb{C}[z]$.  There is a map $\mathbb{C}[z]\rightarrow Sp(S^2,\infty)$ which sends a polynomial to its roots.  This map descends to a homeomorphism ${\mathbb{C}P}^\infty\cong Sp(S^2,\infty)$.

Write $B_\epsilon(0)\subset \mathbb{C}$ for the $\epsilon$-neighborhood of $0\in \mathbb{C}$.  Once and for all, choose a continuous family of diffeomorphisms $\phi_\epsilon:B_\epsilon(0) \xrightarrow{\cong} \mathbb{C}$ which fix a neighborhood of the origin.  This induces a continuous family of diffeomorphisms $\phi_\epsilon^*:(B_\epsilon(0)^*,*)\xrightarrow{\cong}(S^2,\infty)$ from the one-point compactifications.

Once and for all, fix a parallelization of $F_g\setminus x_0$ and a continuous family of Riemannian metric on $F_g\setminus x_0$ parametrized by $\mathcal{J}_g$.  Assume for each $J\in\mathcal{J}_g$ that the metric is such that the injectivity radius of $F_g\setminus x_0$ is bounded away from $0$.  That is, there is an $\epsilon_J>0$ such that for each $p\in F_g\setminus x_0$, an $\epsilon_J$-neighborhood of $p$ is a convex ball.  For $\epsilon>0$ as so, it follows that the exponential map
\[
exp_p:B_{\epsilon}(0)\rightarrow F_g \setminus x_0
\]
is an embedding where $B_\epsilon(0)\subset \mathbb{C}$ is an $\epsilon$-neighborhood of $0\in \mathbb{C}$.  Assume further that this family of metrics is chosen so that for each $J\in \mathcal{J}_g$ the distances are bounded below 
\begin{equation}\label{epsilon}
\epsilon_J< \text{inf}_{i,j\in\mathbb{N}}\{dist(y_J^n(\eta,\zeta),z_J^n(\eta,\zeta))\}
\end{equation}
for all $(\eta,\zeta)\in Div_d$ and $n\in\mathbb{N}$.

Fix a continuous family $\epsilon:\mathcal{J}_g\rightarrow(0,\infty)$ as in the above paragraph.  Denote by $\mathcal{D}^\epsilon\subset \mathcal{D}$ the subfunctor with $\mathcal{D}^\epsilon(n)\subset \mathcal{D}(n)$ the subspace consisting of those triples $(J,\eta,\zeta)$ for which 
\[
\emptyset = (\bigcup_{u\in \eta} B_{\epsilon_J}(u)) \cap  (\bigcup_{v\in \zeta} B_{\epsilon_J}(v)) \subset F_g\setminus x_0.
\]
Because of condition~(\ref{epsilon}), this condition is indeed preserved under the structure maps of the diagram $\mathcal{D}$.  Each of the inclusions $\mathcal{D}^\epsilon(n)\hookrightarrow\mathcal{D}(n)$ in addition to the inclusion $\widehat{Div}^\epsilon \hookrightarrow \widehat{Div}$ is a homotopy equivalence.

Let $n\in\mathbb{N}$.  There is a `scanning map'
\[
S^\epsilon_n:\mathcal{D}^\epsilon(n) \rightarrow Map_*(F_g,{\mathbb{C}P}^\infty\vee{\mathbb{C}P}^\infty)
\]
whose adjoint is given as 
\[
((J,\eta,\zeta),p)\mapsto (exp_p^{-1}(\eta))\vee (exp_p^{-1}(\zeta)),
\]
for $p\neq x_0$,
where $ - $ denotes subtraction in the twice delooped $B^2\mathbb{Z} \cong {\mathbb{C}P}^\infty \cong_{\phi_\epsilon} Sp(B_\epsilon(0)^*,*)$; and for $p=x_0$ by the assignment $((J,\eta,\zeta),x_0)\mapsto \infty\in {\mathbb{C}P}^\infty\vee {\mathbb{C}P}^\infty$ to the base point.  By construction the following diagram commutes
\begin{equation}\label{scan}
\xymatrix{
\mathcal{D}^\epsilon(n)   \ar[rr]   \ar[dr]^{S^\epsilon_n}
&
&
\mathcal{D}^\epsilon(n+1)\ar[dl]^{S^\epsilon_{n+1}}
\\
&
Map_*(F_g,{\mathbb{C}P}^\infty\vee{\mathbb{C}P}^\infty).
&
}
\end{equation}
This gives the data of a map from the colimit
\[
S^\epsilon:\widehat{Div}^\epsilon \rightarrow Map_*(F_g,{\mathbb{C}P}^\infty\vee{\mathbb{C}P}^\infty).
\]
Refer to this map also as a `scanning map'.  Upon the (contractible) choice of a homotopy inverse, from the zig-zag
\[
\widehat{Div}\xleftarrow{\simeq} \widehat{Div}^\epsilon \xrightarrow{S^\epsilon} Map_*(F_g,{\mathbb{C}P}^\infty\vee{\mathbb{C}P}^\infty)
\]
there results a map $\widehat{Div} \xrightarrow{S} Map_*(F_g,{\mathbb{C}P}^\infty\vee{\mathbb{C}P}^\infty)$, referred to as the \textit{scanning map}.  It is straightforward to verify that the homotopy class of $S$ is independent of the choice of the family $\epsilon_J>0$.  Think of $S$ as scanning $F_g$ for either $\eta$ or $\zeta$ with a very zoomed microscope through which, via the parallelization, $F_g$ looks like $\mathbb{C}$.

For $J\in\mathcal{J}_g$ a complex structure, denote by $\widehat{Div}(J)\subset \widehat{Div}$ the fiber over $J$ of the projection map $\widehat{Div}\rightarrow\mathcal{J}_g$.  

\begin{theorem}[Segal (\cite{segal:scanning}~\textsection4)]\label{thm:scanning}
The the scanning map restricts to a homotopy equivalence $S: \widehat{Div}(J) \rightarrow Map_*(F_g, {\mathbb{C}P}^\infty\vee {\mathbb{C}P}^\infty)$.
\end{theorem}

\begin{cor}\label{scan-equivalence}

The scanning map 
\[
S:\widehat{Div}\rightarrow \mathcal{J}\times Map_*(F_g,{\mathbb{C}P}^\infty\vee{\mathbb{C}P}^\infty)
\]
homotopy equivalence.  

\end{cor}

Note that $S$ restricts on the component $S:\widehat{Div}_d \rightarrow Map^d_*(F_g, {\mathbb{C}P}^\infty\vee {\mathbb{C}P}^\infty)$ to a map whose target is the space of maps with bi-degree $(d,d)$.  By Corollary~\ref{scan-equivalence} this restriction is a homotopy equivalence.

\subsection{Comparing Sequences}

Observe that the  connected component of the constant map $Map^0_*(F_g,{\mathbb{C}P}^\infty)$ as well as the $g$-torus $T_0$ are both models for $K(\mathbb{Z}^g,1)$.  So, abstractly, there is a homotopy equivalences $D:T_0\rightarrow Map_*^0(F_g,{\mathbb{C}P}^\infty)$.  In fact, as in (\cite{segal:scanning}, \textsection4), it is possible to choose the homotopy equivalence $D$ via Poincare' duality such that there results a homotopy commutative diagram
\[
\xymatrix{
\mathcal{J}_g^d({\mathbb{C}P}^1)_*  \ar[d] \ar[r]
&
\mathcal{J}_g\times Map_*^d(F_g,S^2)  \ar[d]
\\
\widehat{Div}_d    \ar[d]^{\mathcal{I}} \ar[r]^-S
&
\mathcal{J}_g\times Map^d_*(F_g, {\mathbb{C}P}^\infty\vee {\mathbb{C}P}^\infty)  \ar[d]^{sub}
\\
T_0     \ar[r]^-D
&
Map^0_*(F_g, {\mathbb{C}P}^\infty\vee {\mathbb{C}P}^\infty)
}
\]
where the indicated map is induced by subtraction in ${\mathbb{C}P}^\infty\cong B^2\mathbb{Z}$ so that the right vertical sequence is the fibration  induced from the sequence in~(\ref{eq:fibration}).

Moreover, imitating Segal's lines (\cite{segal:scanning}, Lemma $4.7$), one can verify that the resulting map on fibers is the standard inclusion 
\[
\mathcal{J}^d_g({\mathbb{C}P}^1)_*\hookrightarrow \mathcal{J}_g\times Map^d_*(F_g,S^2).
\]
Because $S$ and $D$ are equivalences (Theorem~\ref{thm:scanning}), we conclude the following lemma. 

\begin{lem}\label{lem:homology-fibration-2}
The inclusion 
\[
\mathcal{J}_g^d({\mathbb{C}P}^1)_*\hookrightarrow \mathcal{J}_g\times Map_*^d(F_g,S^2)
\]
induces an isomorphism in $H_q(-)$ for $q<d-2g$.
\end{lem}

\section{Proof of Theorem~\ref{thm:main-theorem}}

\subsection{Unbased Mapping spaces}

We are interested in the unbased mapping spaces $\mathcal{J}^d_g({\mathbb{C}P}^1)$ and $Map^d(F_g,S^2)$. 
\begin{lem}\label{lem:main-lemma}
The inclusion $\mathcal{J}_g^d({\mathbb{C}P}^1)\hookrightarrow Map^d(F_g,S^2)$ induces an isomorphism in $H_q(-)$ for $q<d-2g$.
\end{lem}

\begin{proof}
Consider the base-point evaluation fibrations
\begin{equation}\label{eq:morphism}
\xymatrix{
\mathcal{J}_g^d({\mathbb{C}P}^1)_*  \ar[d] \ar[r]
& 
\mathcal{J}_g\times Map_*^d(F_g,S^2)   \ar[d] 
\\
\mathcal{J}_g^d({\mathbb{C}P}^1)   \ar[d] \ar[r]
&
\mathcal{J}_g\times Map^d(F_g,S^2)    \ar[d]
\\
{\mathbb{C}P}^1     \ar[r]
& 
S^2.
}
\end{equation}

There results a morphism of Leray-Serre $H_*$-spectral sequences.  Lemma~\ref{lem:homology-fibration-2} shows that this functor induces an isomorphism on the $E^2_{p,q}$ page for $q<d-2g$.  We thus have an isomorphism on the $E^r_{p,q}$ pages for all $r$ when $p+q<d-2g$.  In particular, there is an isomorphism on the $E^\infty_{p,q}$ page for $p+q<d-2g$.

We are now confronted with the common issue of concluding the middle horizontal map in diagram~\ref{eq:morphism} induces an isomorphism in $H_*$, for $*=p+q<d-2g$, knowing it is an isomorphism on the filtration quotients $E^\infty_{p,q}$.  For this one uses induction on the filtration degree; the inductive step is clear in light of the 5-lemma.  
%Alternatively, because we are working over the rationals, all extension problems can be solved.

\end{proof}

\subsection{The action of $Diff^+_g$}

Our ultimate goal being to understand moduli space, we now want a version of Lemma~\ref{lem:main-lemma} which is quotiented by the action of $Diff^+_g$.

\begin{theorem}[Main theorem ($n=1$)]\label{thm:main-theorem-1}
The map $\mathcal{M}^d_g({\mathbb{C}P}^1) \rightarrow \mathcal{MT}^d_g({\mathbb{C}P}^1)$ induced by the natural inclusion induces an isomorphism in $H_q(-)$ for $q< d-2g$.
\end{theorem}

\begin{proof}
There is the following morphism of fibration sequences

\[
\xymatrix{
\mathcal{J}^d_g({\mathbb{C}P}^1) \ar[r]^-{\cong_{H_{<d-2g}}} \ar[d]  
&  
\mathcal{J}_g\times Map^d(F,S^2) \ar[d] 
\\
EDiff^+_g\times_{Diff^+_g}\mathcal{J}^d_g({\mathbb{C}P}^1) \ar[r] \ar[d] 
&
EDiff^+_g\times_{Diff^+_g} (\mathcal{J}_g\times Map^d(F,S^2)) \ar[d]  
\\
BDiff^+_g \ar[r]^{id} 
& 
BDiff^+_g
}
\]

Lemma~\ref{lem:main-lemma} shows that the top horizontal map is an isomorphism in $H_q(-)$ for $q<d-2g$.  Recognizing the total spaces in the above diagram as the \textit{homotopy} moduli spaces, the same spectral sequence argument proving Lemma~\ref{lem:main-lemma} finishes the proof.

\end{proof}

\subsection{Maps to ${\mathbb{C}P}^n$ for $n\geq 1$}

In this subsection we will sketch the idea for how to deal with maps to $\cp$ for $n>1$.  The argument is nearly identical to that explained in~\cite{segal:scanning}.

\begin{theorem}[Main Theorem]\label{thm:main-theorem'}
The map $\mathcal{M}_g^d(\cp)\rightarrow \mathcal{MT}_g^d(\cp)$ induces an isomorphism in $H_q(-)$ for $q<(d-2g)(2n-1)$.
\end{theorem}

\begin{proof}[Proof (sketch)]

The proof is analogous to the proof of Theorem~\ref{thm:main-theorem-1}.  We will sketch the appropriate modifications.  Details will be left to the interested reader.

A holomorphic map $h\in \mathcal{J}^d_g({\mathbb{C}P}^1)_*$ is a based meromorphic function and is thus determined by its poles and zeros.  Equivalently, we could regard $h$ as a pair $(r_0,r_1)$ of degree $d$ polynomial functions on $(F_g,J)$ whose zeros are disjoint.  This is the essence of Abel's theorem that we used earlier to regard $\mathcal{J}^d_g({\mathbb{C}P}^1)_*$ as a subspace of $\mathcal{J}_g\times Div_d$.  Now, think of $h\in \mathcal{J}^d_g({\mathbb{C}P}^n)_*$ similarly as an $(n+1)$-tuple $(r_0,...,r_n)$ of degree $d$ polynomials with $\bigcap_{i=0}^n \{r_i=0\} = \emptyset$.  One would thus generalize $Div_d$ to $Div_d^{(n)}$ consisting of $(n+1)$-tuples of positive divisors which are $(n+1)$-wise disjoint.

A modification of Abel's theorem shows that such an $(n+1)$-tuple comes from a map into $\cp$ if and only if each divisor has the same image in the Jacobian variety and, moreover, the fiber of the similarly defined integration map $\mathcal{I}:\mathcal{J}_g\times Div_d^{(n)}\rightarrow T_0^n$ is the space of based holomorphic maps $\mathcal{J}^d_g({\mathbb{C}P}^n)_*$.  The surrounding analogous constructions are similar.

On the homotopy theoretic side, one constructs the space 
\[
W_{n+1}({\mathbb{C}P}^\infty) \subset \prod^{n+1} {\mathbb{C}P}^\infty  
\]
consisting of $(n+1)$-tuples with at least one entry the base point of ${\mathbb{C}P}^\infty$.  Indeed, $W_2({\mathbb{C}P}^\infty) = {\mathbb{C}P}^\infty \vee {\mathbb{C}P}^\infty$.  Similar to sequence~\ref{eq:fibration} is the (homotopy) fibration sequence
\[
{\mathbb{C}P}^n\rightarrow W_{n+1}({\mathbb{C}P}^\infty) \rightarrow \prod^n {\mathbb{C}P}^\infty.
\]
There is a similar scanning map used to compare $Div_d^{(n)}$ to $Map_*^d(F_g,W_{n+1}({\mathbb{C}P}^\infty))$. 
Details are left to the interested reader.

\end{proof}

\section{Variants of the main theorem}\label{variants}

In this final section we outline two simple variations on the main theorem~\ref{thm:main-theorem}.

\subsection{Marked points}

Choose $k$ points $\{r_i\}\subset F_g$.  Let $Diff^+_{g,k}$ denote the subgroup of $Diff^+_g$ consisting of those diffeomorphisms of $F_g$ which fix each $r_i$ for $1\leq i \leq k$.  Define the \textit{moduli space of degree $d$ marked holomorphic curves of genus $g$ in ${\mathbb{C}P}^n$} as the orbit space
\[
\mathcal{M}^d_{g,k}({\mathbb{C}P}^n): = \mathcal{J}^d_g({\mathbb{C}P}^n)//Diff^+_{g,k}.
\]
Do similarly for the topological counterpart $\mathcal{MT}^d_{g,k}({\mathbb{C}P}^n)$.

\begin{theorem}\label{thm:marks}
The standard map
\[
\mathcal{M}^d_{g,k}({\mathbb{C}P}^n) \rightarrow \mathcal{MT}^d_{g,k}({\mathbb{C}P}^n)
\]
induces an isomorphism in $H_q(-)$ for $q<(d-2g)(2n-1)$.
\end{theorem}

\begin{proof}

The proof is nearly identical to that of Theorem~\ref{thm:main-theorem'}.

\end{proof}

\subsection{Moduli of singular Riemann surfaces}\label{moduli-singular}

We define a notion of a singular surface specific to our purposes.  Let $F$ be a compact Hausdorff space.  A structure of a singular surface on $F$ is the data of a finite subset $P\subset F$ and the structure of an oriented smooth surface on $F\setminus P$.  Such an $F$ endowed with this structure will be referred to as a \textit{singular surface}.  Such singular surfaces are characterized as quotient spaces
\[
F = (\amalg_1^n F_i)/\sim.
\]
where each $F_i$ is a smooth oriented surface with a finite collection of marked points $\{p_{i_k}\}_1^{n_i}\subset F_i$, which are identified $p_{i_k}\sim p_{j_l}$ in some way.  Refer to the data $(F_i,\{p_{i_k}\}_1^{n_i})$ as a \textit{normalization} of $F$; note that this data is unique.  
Refer to the subset $P\subset F$ as the \textit{nodes} and denote the cardinality $n(F):= \lvert P\rvert$.  Define the \textit{genus} of $F$ as the number $g(F)$ given by
\[
\chi(F) = 2-2g(F)+n(F)
\]
Say a singular surface is \textit{irreducible} if its normalization consists of a single connected surface.

A complex structure $J$ on such a singular surface $F$ is the data of a complex structure $J_i$ on each component $F_i$ of its normalization.  Write $\mathcal{J}_F$ for the space of such complex structures with the apparent topology.  Refer to a pair $(F,J)$ as a \textit{singular Riemann surface}.  A holomorphic map from $(F,J)$ to a complex manifold $Y$ is the data of holomorphic maps $(F_i,J_i)\xrightarrow{h_i} Y$ such that $h_i(p_{i_k}) = h_j(p_{j_l})$ whenever $p_{i_k}\sim p_{j_l}$ in $F$.  The degree of such a map is the sum $\Sigma_i deg( h_i)\in H_2(Y)$.  Let $Hol^\alpha((F,J),Y)$ be the space of degree $\alpha$ holomorphic maps from $(F,J)$ into $Y$, topologized in the obvious way.

\begin{remark}
In what follows, most of the statements about singular surfaces are stated only for \textit{irreducible} singular surfaces.  One could make statements about non-irreducible surfaces provided one specifies the degree of maps on each component of the normalization.
\end{remark}

In~\cite{segal:scanning}, Segal proves the following
\begin{theorem}[Segal]
Let $F$ be an irreducible singular surface.  The standard inclusion
\[
Hol^d((F,J),{\mathbb{C}P}^n)\hookrightarrow  Map^d(F,{\mathbb{C}P}^n)
\]
induces an isomorphism in $H_q(-)$ for $q<(d-2(g(F)-n(F)+1))(2n-1)$.
\end{theorem}

The idea of the proof is that a holomorphic map $(F,J)\rightarrow \cp$ from a singular surface \textit{is} a holomorphic map $(F_1,J_1)\rightarrow \cp$ with conditions on the marked points $\{p_{1_k}\}$.  The theorem follows from Segal's previous work when $F$ is not singular (i.e., the set of nodes is empty, $P=\emptyset$).

For $F$  a singular surface, write $Diff^+(F)\subset Diff^+(\coprod F_i)$ as the topological subgroup of those diffeomorphisms $\phi$ of $\coprod F_i$ for which $\phi(p_{i_k}) \sim p_{i_k}$.  The group $Diff^+(F)$ acts on the space of pairs 
\[
\mathcal{J}^d_{[F]}({\mathbb{C}P}^n):= \{(J,h)\mid  J\in \mathcal{J}_F \text{ and }h\in Hol^d((F,J),Y)\}
\]
by pulling back complex structures and precomposing with holomorphic maps.

Let $F$ be a singular surface.  Define the \textit{moduli space of type $[F]$} to be the homotopy orbit space
\begin{equation}\label{singular}
\mathcal{M}^d_{[F]}(\cp):=\mathcal{J}^d_{[F]}({\mathbb{C}P}^n)//Diff^+(F).
\end{equation}
Define similarly the \textit{topological moduli space of type $[F]$} as
\[
\mathcal{MT}^d_{[F]}(\cp):=(\mathcal{J}_F\times Map(F,\cp))//Diff^+(F).
\]

\begin{theorem}\label{thm:singular-main-theorem}
Let $F$ be an irreducible singular surface.  The standard map
\[
\mathcal{M}^d_{[F]}(\cp)\rightarrow \mathcal{MT}^d_{[F]}(\cp) \simeq  Map(F,{\mathbb{C}P}^n)//Diff^+(F)
\]
induces an isomorphism in $H_q(-)$ for $q<(d-2(g(F)-n(F)+1))(2n-1)$.
\end{theorem}

\begin{proof}
The proof will merely be outlined.  As demonstrated with the proof of Theorem~\ref{thm:main-theorem}, it is possible to follow the lines of Segal for singular Riemann surfaces while keeping track of the complex structure on such Riemann surfaces as a variable.  This leads to a statement analogous to Lemma~\ref{lem:main-lemma} for singular surfaces.  The same spectral sequence argument from the proof of Theorem~\ref{thm:main-theorem-1} is in place to account for quotienting by the action of $Diff^+(F)$.  

\end{proof}

\begin{remark}

This remark serves to report a curiosity of the author.  Consider the compactified moduli space $\overline{\mathcal{M}}^d_g({\mathbb{C}P}^n)$ of Gromov-Witten theory (see \cite{mcduff-salamon:holomorphic-curves} for a general reference).  An $S$-point of this space (stack) is in particular a family of nodal Riemann surfaces over $S$ together with a degree $d$ holomorphic maps into ${\mathbb{C}P}^n$.  There is a \textit{topological} counterpart $\overline{\mathcal{MT}}^d_g({\mathbb{C}P}^n)$ which replaces holomorphic maps by continuous maps. Write 
\[
\overline{\mathcal{M}}^{irr}_g({\mathbb{C}P}^n)\hookrightarrow \overline{\mathcal{M}}_g({\mathbb{C}P}^n)
\]
for the subspace (substack)
consisting of irreducible nodal surfaces, and similarly for the topological counterpart.  There is a natural stratification of $\overline{\mathcal{M}}^{irr}_g({\mathbb{C}P}^n)$ in terms of the number of nodes of surfaces (see \cite{kock-vainsencher:kontsevich} for example).  It is a naive observation that the open strata are moduli spaces of the form~(\ref{singular}).  One might expect then from Theorem~\ref{thm:singular-main-theorem} that the apparent comparison
\[
\overline{\mathcal{M}}^{d,{irr}}_g({\mathbb{C}P}^n)\hookrightarrow \overline{\mathcal{MT}}^{d,irr}_g({\mathbb{C}P}^n)
\]
is then an isomorphism in $H_q$ for $q<<g<<d$.  A needed ingredient is a result on the existence of regular neighborhoods of the closed strata of $\overline{\mathcal{M}}^{irr}_g({\mathbb{C}P}^n)$.

\end{remark}


\begin{thebibliography}{McC01}

\bibitem[Arn70]{arnold:braids}
V.~I. Arnold.
\newblock Certain topological invariants of algebrac functions.
\newblock {\em Trudy Moskov. Mat. Ob\v s\v c.}, 21:27--46, 1970.

\bibitem[CM06]{cohen-madsen:background}
Ralph~L. Cohen and Ib~Madsen.
\newblock Surfaces in a background space and the homology of mapping class
  groups.
\newblock on the ArXiv, 2006.

\bibitem[DL62]{dyer-lashof:operations}
Eldon Dyer and R.~K. Lashof.
\newblock Homology of iterated loop spaces.
\newblock {\em Amer. J. Math.}, 84:35--88, 1962.

\bibitem[DM69]{deligne-mumford:comapct}
P.~Deligne and D.~Mumford.
\newblock The irreducibility of the space of curves of given genus.
\newblock {\em Inst. Hautes \'Etudes Sci. Publ. Math.}, (36):75--109, 1969.

\bibitem[GH94]{griffiths-harris:principles-of-algebraic-geometry}
Phillip Griffiths and Joseph Harris.
\newblock {\em Principles of algebraic geometry}.
\newblock Wiley Classics Library. John Wiley \& Sons Inc., New York, 1994.
\newblock Reprint of the 1978 original.

\bibitem[Har85]{harer:stability}
John~L. Harer.
\newblock Stability of the homology of the mapping class groups of orientable
  surfaces.
\newblock {\em Ann. of Math. (2)}, 121(2):215--249, 1985.

\bibitem[HM98]{harris-morrison:curves}
Joe Harris and Ian Morrison.
\newblock {\em Moduli of curves}, volume 187 of {\em Graduate Texts in
  Mathematics}.
\newblock Springer-Verlag, New York, 1998.

\bibitem[KV99]{kock-vainsencher:kontsevich}
Joachim Kock and Israel Vainsencher.
\newblock {\em A f\'ormula de {K}ontsevich para curvas racionais planas}.
\newblock 22$\sp {\rm o}$ Col\'oquio Brasileiro de Matem\'atica. [22nd
  Brazilian Mathematics Colloquium]. Instituto de Matem\'atica Pura e Aplicada
  (IMPA), Rio de Janeiro, 1999.

\bibitem[McC01]{mccleary:spectral-sequences}
John McCleary.
\newblock {\em A user's guide to spectral sequences}, volume~58 of {\em
  Cambridge Studies in Advanced Mathematics}.
\newblock Cambridge University Press, Cambridge, second edition, 2001.

\bibitem[MS04]{mcduff-salamon:holomorphic-curves}
Dusa McDuff and Dietmar Salamon.
\newblock {\em {$J$}-holomorphic curves and symplectic topology}, volume~52 of
  {\em American Mathematical Society Colloquium Publications}.
\newblock American Mathematical Society, Providence, RI, 2004.

\bibitem[MS76]{mcduff-segal:group-completion}
D.~McDuff and G.~Segal.
\newblock Homology fibrations and the ``group-completion'' theorem.
\newblock {\em Invent. Math.}, 31(3):279--284, 1975/76.

\bibitem[NN57]{newlander-nirenberg:complex-almost-complex}
A.~Newlander and L.~Nirenberg.
\newblock Complex analytic coordinates in almost complex manifolds.
\newblock {\em Ann. of Math. (2)}, 65:391--404, 1957.

\bibitem[Seg79]{segal:scanning}
Graeme Segal.
\newblock The topology of spaces of rational functions.
\newblock {\em Acta Math.}, 143(1-2):39--72, 1979.

\bibitem[Zee57]{zeeman:comparison}
E.~C. Zeeman.
\newblock A proof of the comparison theorem for spectral sequences.
\newblock {\em Proc. Cambridge Philos. Soc.}, 53:57--62, 1957.

\end{thebibliography}
\end{document}